\documentclass[a4paper,12pt]{article}
\usepackage{amsmath}
\usepackage{amsfonts}
\usepackage{amssymb}
\usepackage{amsthm}
\usepackage{bm}
\usepackage{bbm}
\usepackage{mathtools}
\mathtoolsset{showonlyrefs}
\usepackage{mleftright}
\mleftright
\usepackage{xparse}
\usepackage[hidelinks]{hyperref}
\theoremstyle{definition}

\theoremstyle{plain}
\newtheorem{thm}{Theorem}
\newtheorem{lem}{Lemma}

\newtheorem*{conj}{Conjecture}
\theoremstyle{remark}
\newtheorem*{rmk}{Remark}
\newtheorem*{note}{Note}

\newcommand{\ZZ}{\mathbb{Z}}

\newcommand{\RR}{\mathbb{R}}

\DeclarePairedDelimiterX{\preset}[2]{\lbrace}{\rbrace}{#1\,\delimsize\vert\,#2}

\newcommand{\abs}[1]{\left|#1\right|}
\newcommand{\norm}[1]{\left\|#1\right\|}
\newcommand{\floor}[1]{\left\lfloor#1\right\rfloor}

\let\O\relax
\NewDocumentCommand{\O}{E{^_}{{}{}} r()}{O^{#1}_{#2}\left(#3\right)}
\let\o\relax
\NewDocumentCommand{\o}{E{^_}{{}{}} r()}{o^{#1}_{#2}\left(#3\right)}
\NewDocumentCommand{\Omg}{E{^_}{{}{}} r()}{\operatorname{\Omega}^{#1}_{#2}\left(#3\right)}
\let\Re\relax
\NewDocumentCommand{\Re}{d()}{\IfValueTF{#1}{\operatorname{Re}\left(#1\right)}{\operatorname{Re}}}

\NewDocumentCommand{\mmu}{r()}{(\mu\ast\mu)\left(#1\right)}
\NewDocumentCommand{\rad}{r()}{\operatorname{rad}\left(#1\right)}
\let\exp\relax
\NewDocumentCommand{\exp}{d()}{\IfValueTF{#1}{\operatorname{exp}\left(#1\right)}{\operatorname{exp}}}
\let\log\relax
\NewDocumentCommand{\log}{d()}{\IfValueTF{#1}{\operatorname{log}\left(#1\right)}{\operatorname{log}}}
\NewDocumentCommand{\lr}{r()}{\left(#1\right)}
\newcommand{\dd}[1]{\mathop{{\operatorfont{d}}#1}}
\NewDocumentCommand{\indic}{E{^_}{{}{\mathbb{Z}}} r()}{\mathbbm{1}^{#1}_{#2}\left(#3\right)}
\NewDocumentCommand{\dlt}{E{^_}{{}{}} r()}{\delta^{#1}_{#2}\left(#3\right)}
\NewDocumentCommand{\omg}{E{^_}{{}{}} m o m}{\IfValueTF{#4}{\operatorname{\omega}^{#1}_{#2}\left(#3; #4, #5\right)}{\operatorname{\omega}^{#1}_{#2}\left(#3; #5\right)}}
\NewDocumentCommand{\J}{E{^_}{{}{}} m m}{J^{#1}_{#2}\left(#3, #4\right)}
\let\min\relax
\NewDocumentCommand{\min}{d\{\}}{\IfValueTF{#1}{\operatorname{min}\left\{#1\right\}}{\operatorname*{min}}}
\let\max\relax
\NewDocumentCommand{\max}{d\{\}}{\IfValueTF{#1}{\operatorname{max}\left\{#1\right\}}{\operatorname*{max}}}
\allowdisplaybreaks[4]
\title{\textbf{Sums of Apostol's M\"{o}bius functions of order $k$}}
\author{Reo Terada}
\begin{document}
\maketitle
\begin{abstract}
  In 1970, T. M. Apostol introduced the M\"obius function $\mu_{k}$ of order $k$ for all positive integer $k$, as a generalization of the M\"obius function $\mu = \mu_{1}$. For any integer $k \ge 2$, he proved $\sum_{n \le x} \mu_{k}(n) = A_{k} x + \O_{k}(x^{1/k} \log x)$ where $A_{k}$ is a positive constant. 
  In 2001, A. Bege conjectured both the conditional and unconditional estimates for the sum $\sum_{n \le x, (n, q) = 1}\mu_{k}(n)$ for any positive integer $q$. 
  In this paper, we give affirmative solutions to the conditional version of Bege's conjecture completely and the unconditional one partially. We also give a mean square estimate for the error term. 
\end{abstract}
\section{Introduction}
For any complex number $s = \sigma + it$ where $\sigma, t \in \RR$, we let $\zeta(s)$ denote the Riemann zeta-function. 
In 1970, T. M. Apostol~\cite{A} introduced the following generalization of the M\"{o}bius function $\mu$.
For each fixed positive integer $k$, the M\"{o}bius function $\mu_{k}$ of order $k$ is defined by
\begin{equation}
  \mu_{k}(n) \coloneq
  \begin{cases*}
    1        & if $n = 1$, \\
    (-1)^{r} & if $n = (p_{1} p_{2} \cdots p_{r})^{k} \prod_{i>r} p_{i}^{a_{i}}$ with $0\le a_{i}<k$, \\
    0        & if $p^{k + 1} \mid n$ for some prime $p$, \\
    1        & otherwise. 
  \end{cases*} \label{eq:muk}
\end{equation}
Here $p_{1}, p_{2}, p_{3}, \ldots$ are distinct primes. 
In other words, $\mu_{k}(n)$ vanishes if $n$ is divisible by $p^{k + 1}$ for some prime $p$;  
otherwise, $\mu_{k}(n)$ is $1$ unless the prime factorization of $n$ contains the $k$-th powers of exactly $r$ distinct primes, 
in which case $\mu_{k}(n) = (-1)^{r}$. 
If $k = 1$, $\mu_{k}(n)$ is the usual M\"{o}bius function, that is $\mu_{1}(n) = \mu(n)$. 
The asymptotic formula for the summatory function $\sum_{n\leq x}\mu_{k}(n)$ was first considered by Apostol~\cite{A}. 
For any integer $k \ge 2$ and any real number $x \ge 2$, he proved that
\begin{equation}
  \sum_{n \le x}\mu_{k}(n) = A_{k}x + E_{k}(x) \label{eq:A}
\end{equation}
holds, where the error term $E_{k}(x)$ is $\O_{k}(x^{1/k} \log x)$, and the constant $A_{k}$ is given by  
\begin{equation}
  A_{k} = \prod_{p} \lr(1 - 2 p^{-k} + p^{-(k + 1)}). \label{eq:Ak}
\end{equation}
In 1977, under the Riemann Hypothesis, D. Suryanarayana~\cite{S} showed that 
\begin{equation}
E_{k}(x) = \O(x^{4k/(4k^{2} + 1)}\exp(C_{0} \frac{\log x}{\log\log x})), \label{eq:S}
\end{equation}
where $C_{0}$ is an absolute positive constant.

We now consider a more general case by adding the condition that $n$ is coprime to $q$. 
Let $x \ge 1$, we write for integers $k \ge 2$ and $q \ge 1$, 
\[
\sum_{\substack{n \le x \\ (n, q) = 1}}\mu_{k}(n) = A_{k, q} x + E_{k, q}(x), 
\]
where $E_{k, q}(x)$ denotes the error term. 
In 2001, A. Bege~\cite{B} proposed the following conjecture on the size of $E_{k, q}(x)$. 
\begin{conj}\label{conj:B}
  For any real number $x \ge 3$ and any integers $k \ge 2$ and $q \ge 1$, we have
  \begin{equation}
    E_{k, q}(x) \ll \theta(q) x^{1/k} \exp(-D \frac{(\log x)^{3/5}}{(\log\log x)^{1/5}}) \label{eq:B-uncond} 
  \end{equation}
  where
  \begin{equation}
    A_{k, q} \coloneq \frac{\varphi(q)}{q}\prod_{p \nmid q}\lr(1 - 2 p^{-k} + p^{-(k + 1)}), \label{eq:Akq}
  \end{equation}
  $\varphi(q)$ is Euler's totient function, $\theta(q)$ is the number of positive squarefree divisors of $q$ and $D$ is an absolute positive constant. 
  In particular, when $q = 1$ the conjecture is
  \begin{equation}
    E_{k}(x) \ll x^{1/k} \exp(-D \frac{(\log x)^{3/5}}{(\log\log x)^{1/5}}). \label{eq:B-uncond-1}
  \end{equation}
  
  If the Riemann Hypothesis is true, then the estimate is improved to
  \begin{equation}
    E_{k, q}(x) \ll \theta(q) x^{2/(2 k + 1)} \exp(A \frac{\log x}{\log\log x}) \label{eq:B-RH}
  \end{equation}
  for any real number $x\ge 3$ and any integer $k\ge 2$ and $q \ge 1$. Here $A$ is an absolute positive constant. 
  In particular, when $q = 1$ the conjecture is
  \begin{equation}
    E_{k}(x) \ll x^{2/(2 k + 1)}\exp(A \frac{\log x}{\log\log x}). \label{eq:B-RH-1}
  \end{equation}
\end{conj}

In 2023, D. Banerjee, Y. Fujisawa, T. M. Minamide and Y. Tanigawa~\cite{BFMT} improved Apostol's result $E_{k}(x) \ll_{k} x^{1/k} \log x$ to
\begin{equation}
  E_{k}(x) \ll_{k} x^{1/k} \exp(-D_{0} k^{-13/5} \frac{(\log x)^{3/5}}{(\log\log x)^{1/5}}) \label{eq:BFMT}
\end{equation}
unconditionally, where $D_{0}$ is an absolute positive constant. 
This result solves the unconditional part of Bege's conjecture when $q = 1$ in a version which the $O$-constant and the constant $D$ in \eqref{eq:B-uncond-1} may depend on $k$. 

Recently, in 2025, under the Riemann Hypothesis, G. Martin and C. H. Yip~\cite{MY} show that
\begin{equation}
  E_{k}(x) \ll_{k, \varepsilon} x^{1/(k + 1) + \varepsilon} \label{eq:MY}
\end{equation}
for any real number $x \ge 1, \varepsilon > 0$ and any integer $k \ge 2$. 
This result gives an affirmative solution to the conditional part of Bege's conjecture with $q = 1$ in a version where the $O$-constant in \eqref{eq:B-RH-1} may depend on $k$. 
They also give the following oscillation result: 
\begin{equation}
  E_{k}(x) = B_{k} x^{1/(k+1)} + \Omg_{\pm}(x^{1/(2 k)} \log x), \quad x \to +\infty. \label{eq:MY-Omg}
\end{equation}
where
\begin{equation}
  B_{k} \coloneq \frac{\zeta(1/(k + 1))}{\zeta^{2}(k/(k + 1))}\prod_{p}\frac{\lr(1 - 2 p^{-k/(k + 1)} + p^{-1})\lr(1 - p^{-1})}{\lr(1 - p^{-k/(k + 1)})^{2}}. \label{eq:Bk}
\end{equation}

The purpose of this paper is to generalize Banerjee, Fujisawa, Minamide and Tanigawa's result~\eqref{eq:BFMT} and Martin and Yip's result~\eqref{eq:MY} to the case $q \ge 1$. we also estimate the mean square of the error term. 
\begin{thm}\label{thm:uncond}
  For any real number $x \ge 3$ and any integers $k \ge 2$ and $q \ge 1$, we have
  \begin{equation}
    E_{k, q}(x) \ll_{k} \theta(q) x^{1/k} \exp(-D_{0} k^{-8/5} \frac{(\log x)^{3/5}}{(\log\log x)^{1/5}}), \label{main}
  \end{equation}
  where $D_{0}$ is an absolute positive constant. 
\end{thm}
\begin{rmk}
  Theorem~\ref{thm:uncond} answers a version of the unconditional part of Bege's conjecture in which the $O$-constant and the constant $D$ in \eqref{eq:B-uncond} may depend on $k$ in the affirmative. 
  If $q = 1$, it also slightly improves Banerjee, Fujisawa, Minamide and Tanigawa's result~\eqref{eq:BFMT}. 
\end{rmk}
\begin{thm}\label{thm:RH}
  Assume that the Riemann Hypothesis is true. 
  For any real number $x\ge 3$ and any integers $k\ge 2$ and $q \ge 1$, 
  \begin{equation}
    E_{k, q}(x) \ll \theta(q) x^{1/(k + 1)}\exp(C_{0} \frac{\log x}{\log\log x}) \label{mainRH}
  \end{equation}
  holds. Here $C_0$ is an absolute positive constant. 
\end{thm}
\begin{rmk}
  Theorem~\ref{thm:RH} completely solves the conditional part of Bege's conjecture~\eqref{eq:B-RH}. 
  If $q = 1$, it also gives a little refinement to Martin and Yip's result~\eqref{eq:MY}. 
\end{rmk}

Next, we consider the mean square of the function 
\begin{equation}
  \Xi_{k, q}(x) \coloneq E_{k, q}(x) - B_{k, q} x^{1/(k + 1)}
\end{equation}
where
\begin{equation}
  \begin{aligned}
    B_{k, q} &\coloneq \frac{\varphi(q) \J{1/(k + 1)}{q} q^{(k - 2)/(k + 1)} \zeta(1/(k + 1))}{\J^{2}{k/(k + 1)}{q} \zeta^{2}(k/(k + 1))}\\
    &\phantom{{}\coloneq{}}\quad \times \prod_{p \nmid q} \frac{\lr(1 - 2 p^{-k/(k + 1)} + p^{-1})\lr(1 - p^{-1})}{\lr(1 - p^{-k/(k + 1)})^{2}}
  \end{aligned}\label{eq:Bkq}
\end{equation}
and
\begin{equation}
  \J{s}{q} \coloneq q^{s} \prod_{p \mid q} \lr(1 - p^{-s}). \label{eq:jordan}
\end{equation}
In other words, $\Xi_{k, q}(x)$ measures the discrepancy of the size of the error term $E_{k, q}(x)$ from its expected size. 
\begin{thm}\label{thm:sq}
  For any real number $T \ge 3$ and any integers $k \ge 2$ and $q \ge 1$, we have
  \begin{equation}
    \int^{T}_{1} \Xi_{k, q}^{2}(x) \dd{x} \ll C_{k, q}^{2} T^{1 + 1/k} \exp(G_{0}^{k} \lr(\log T)^{G_{1} k/(G_{1} k + 1)}) \label{eq:sq}
  \end{equation}
  where
  \begin{equation}
    C_{k, q} \coloneq \prod_{p \mid q} \frac{\lr(1 + p^{-1/(2 k)}) \lr(1 + p^{-(k + 1)/(2 k)})}{\displaystyle \lr(1 - p^{-1/2})^{2} \lr(1 - p^{-1}) \lr(1 - p^{-(k + 1)/k})} \label{eq:Ckn}
  \end{equation}
  and $G_{0}$ and $G_{1}$ are absolute positive constants. 
\end{thm}
\begin{note}
  For any real number $\varepsilon > 0$ and any integer $k \ge 2$ and $q \ge 1$, we have
  \[
  C_{k, q} \ll_{\varepsilon} \theta^{1 + \varepsilon}(q)
  \]
  and
  \[
  C_{k, q} \ll_{k, \varepsilon} \theta^{\varepsilon}(q). 
  \]
\end{note}
\section{Lemmas}
\begin{lem}\label{lem:decomp1}
  For any integer $k \ge 2$, we decompose the function $\mu_{k}$ as follows. 
  \begin{equation}
    \mu_{k} = f_{k} \ast c_{k} \label{eq:decomp1}
  \end{equation}
  where
  \begin{equation}
    f_{k}(n) \coloneq \sum_{d^{k} \mid n}\mmu(d). \label{eq:fk}
  \end{equation}
  Then the function $c_{k}$ satisfies
  \begin{equation}
    \sum_{n = 1}^{\infty}\frac{c_{k}(n)}{n^{s}} = \prod_{p}\frac{1 - 2 p^{-k s} + p^{-(k + 1) s}}{\lr(1 - p^{-k s})^{2}} \label{eq:ck}
  \end{equation}
  for $\Re(s) > 1/(k + 1)$. The Dirichlet series and the Euler product in \eqref{eq:ck} converge absolutely for $\Re(s) > 1/(k + 1)$. 
\end{lem}
\begin{proof}
  See the proof of Lemma~2.5 in \cite{BFMT}. 
\end{proof}
\begin{lem}\label{lem:decomp2}
  For any integers $k \ge 2$ and $n \ge 1$, we have
  \begin{equation}
    c_{k}(n) = \sum_{d^{k + 1} \mid n} b_{k}\lr(\frac{n}{d^{k + 1}}) \label{eq:ckbk}
  \end{equation}
  where
  \begin{equation}
    \sum_{n = 1}^{\infty} \frac{b_{k}(n)}{n^{s}} = \prod_{p} \frac{\lr(1 - 2 p^{-k s} + p^{-(k + 1) s})\lr(1 - p^{-(k + 1) s})}{\lr(1 - p^{-k s})^{2}} \label{eq:bk}
  \end{equation}
  for $\Re(s) > 1/(2 k)$. Both sides of Equation \eqref{eq:bk} converge absolutely for $\Re(s) > 1/(2k)$. In addition, 
  \begin{equation}
    \sum_{n = 1}^{\infty}\frac{\abs{b_{k}(n)}}{n^{1/(k + 1)}} \ll 1 \label{eq:absbk}
  \end{equation}
  holds. 
\end{lem}
\begin{proof}
  Since
  \begin{align*}
    \frac{\lr(1 - 2 p^{-k s} + p^{-(k + 1) s})\lr(1 - p^{-(k + 1) s})}{\lr(1 - p^{-k s})^{2}} &= 1 - p^{-2 k s} \lr(\frac{1 - p^{-s}}{1 - p^{-k s}})^{2}, 
  \end{align*}
  we obtain
  \[
  \abs{b_{k}(n)} \le b^{\ast}_{k}(n)
  \]
  where the function $b^{\ast}_{k}$ is defined by the following equation: 
  \begin{equation}
    \sum_{n = 1}^{\infty} \frac{b^{\ast}_{k}(n)}{n^{s}} = \prod_{p} \lr(1 + p^{-2 k s} \lr(\frac{1 + p^{-s}}{1 - p^{-k s}})^{2}). 
  \end{equation}
  For $\sigma > 1/(2k)$, we see that
  \[
  \lr(\frac{1 + p^{-\sigma}}{1 - p^{-k \sigma}})^{2} \ll 1. 
  \]
  We then immediately obtain Lemma~\ref{lem:decomp2}. 
\end{proof}
\begin{lem}\label{lem:decomp3}
  For any integer $k \ge 2$, we have
  \[
  \sum_{n = 1}^{\infty} \frac{b_{k}(n)}{n^{s}} = \frac{\zeta^{2}((2 k + 1) s)}{\zeta(2 k s)\zeta((2 k + 2) s)} D_{k}(s)
  \]
  for $\Re(s) > 1/(2 k)$. Here
  \begin{equation}
    \begin{split}
      D_{k}(s) &= \sum_{n = 1}^{\infty} \frac{d_{k}(n)}{n^{s}}\\
      &= \prod_{p} \frac{\lr(1 - 2 p^{-k s} + p^{-(k + 1)s}) \lr(1 - p^{-(k + 1) s}) \lr(1 - p^{-(2 k + 1) s})^{2}}{\lr(1 - p^{-k s})^{2} \lr(1 - p^{-2 k s}) \lr(1 - p^{-(2 k + 2) s})}
    \end{split} \label{eq:dk}
  \end{equation}
  for $\Re(s) > 1/(3 k)$. The Dirichlet series and the Euler product in \eqref{eq:dk} converge absolutely for $\Re(s) > 1/(3 k)$. Furthermore, 
  \[
  \sum_{n = 1}^{\infty} \frac{\abs{d_{k}(n)}}{n^{1/(2 k)}} \ll 1
  \]
  holds. 
\end{lem}
\begin{proof}
  Since
  \begin{align*}
    &\frac{\lr(1 - 2 p^{-k s} + p^{-(k + 1)s}) \lr(1 - p^{-(k + 1) s}) \lr(1 - p^{-(2 k + 1) s})^{2}}{\lr(1 - p^{-k s})^{2} \lr(1 - p^{-2 k s}) \lr(1 - p^{-(2 k + 2) s})}\\
    &\quad = 1 - p^{-3 k s} \lr(1 - p^{-k s})^{-2} \lr(1 - p^{-2 k s})^{-1} \lr(1 - p^{-(2 k + 2) s})^{-1}\\
    &\quad \phantom{{} = {}} \quad \times \left(\vphantom{2 - 4 p^{-s} + 2 p^{-2 s} - p^{-k s} + 3 p^{-(k + 2) s} - 2 p^{-(k + 3) s} + p^{-(3 k + 2)} - 2 p^{-(3 k + 3) s} + p^{-(3 k + 4) s}} 2 - 4 p^{-s} + 2 p^{-2 s} - p^{-k s} + 3 p^{-(k + 2) s}\right.\\
    &\quad \phantom{{} = {}} \quad \phantom{{} \times \left(\vphantom{2 - 4 p^{-s} + 2 p^{-2 s} - p^{-k s} + 3 p^{-(k + 2) s} - 2 p^{-(k + 3) s} + p^{-(3 k + 2)} - 2 p^{-(3 k + 3) s} + p^{-(3 k + 4) s}}\right.}\quad \left.{} - 2 p^{-(k + 3) s} + p^{-(3 k + 2)} - 2 p^{-(3 k + 3) s} + p^{-(3 k + 4) s} \vphantom{2 - 4 p^{-s} + 2 p^{-2 s} - p^{-k s} + 3 p^{-(k + 2) s} - 2 p^{-(k + 3) s} + p^{-(3 k + 2)} - 2 p^{-(3 k + 3) s} + p^{-(3 k + 4) s}}\right), 
  \end{align*}
  we find that
  \[
  \abs{d_{k}(n)} \le d^{\ast}_{k}(n)
  \]
  where the function $d^{\ast}_{k}$ is defined by the following equation: 
  \begin{align*}
    &\sum_{n = 1}^{\infty} \frac{d^{\ast}_{k}(n)}{n^{s}}\\
    &\quad = \prod_{p} \left(\vphantom{1 + p^{-3 k s} \lr(1 - p^{-k s})^{-2} \lr(1 - p^{-2 k s})^{-1} \lr(1 - p^{-(2 k + 2) s})^{-1} \left(2 + 4 p^{-s} + 2 p^{-2 s} + p^{-k s} + 3 p^{-(k + 2) s} + 2 p^{-(k + 3) s} + p^{-(3 k + 2)} + 2 p^{-(3 k + 3) s} + p^{-(3 k + 4) s}\right)} 1 + p^{-3 k s} \lr(1 - p^{-k s})^{-2} \lr(1 - p^{-2 k s})^{-1} \lr(1 - p^{-(2 k + 2) s})^{-1}\right.\\
    &\quad \phantom{{} = \prod_{p} \left(\vphantom{1 + p^{-3 k s} \lr(1 - p^{-k s})^{-2} \lr(1 - p^{-2 k s})^{-1} \lr(1 - p^{-(2 k + 2) s})^{-1} \left(2 + 4 p^{-s} + 2 p^{-2 s} + p^{-k s} + 3 p^{-(k + 2) s} + 2 p^{-(k + 3) s} + p^{-(3 k + 2)} + 2 p^{-(3 k + 3) s} + p^{-(3 k + 4) s}\right)}\right.} \quad \times \left(\vphantom{2 + 4 p^{-s} + 2 p^{-2 s} + p^{-k s} + 3 p^{-(k + 2) s} + 2 p^{-(k + 3) s} + p^{-(3 k + 2)} + 2 p^{-(3 k + 3) s} + p^{-(3 k + 4) s}} 2 + 4 p^{-s} + 2 p^{-2 s} + p^{-k s} + 3 p^{-(k + 2) s} + 2 p^{-(k + 3) s}\right.\\
    &\quad \phantom{{} = \prod_{p} \left(\vphantom{1 + p^{-3 k s} \lr(1 - p^{-k s})^{-2} \lr(1 - p^{-2 k s})^{-1} \lr(1 - p^{-(2 k + 2) s})^{-1} \left(2 + 4 p^{-s} + 2 p^{-2 s} + p^{-k s} + 3 p^{-(k + 2) s} + 2 p^{-(k + 3) s} + p^{-(3 k + 2)} + 2 p^{-(3 k + 3) s} + p^{-(3 k + 4) s}\right)}\right.} \quad \phantom{{} \times \left(\vphantom{2 + 4 p^{-s} + 2 p^{-2 s} + p^{-k s} + 3 p^{-(k + 2) s} + 2 p^{-(k + 3) s} + p^{-(3 k + 2)} + 2 p^{-(3 k + 3) s} + p^{-(3 k + 4) s}}\right.} \quad \left.\left.{} + p^{-(3 k + 2)} + 2 p^{-(3 k + 3) s} + p^{-(3 k + 4) s} \vphantom{2 + 4 p^{-s} + 2 p^{-2 s} + p^{-k s} + 3 p^{-(k + 2) s} + 2 p^{-(k + 3) s} + p^{-(3 k + 2)} + 2 p^{-(3 k + 3) s} + p^{-(3 k + 4) s}}\right) \vphantom{1 + p^{-3 k s} \lr(1 - p^{-k s})^{-2} \lr(1 - p^{-2 k s})^{-1} \lr(1 - p^{-(2 k + 2) s})^{-1} \left(2 + 4 p^{-s} + 2 p^{-2 s} + p^{-k s} + 3 p^{-(k + 2) s} + 2 p^{-(k + 3) s} + p^{-(3 k + 2)} + 2 p^{-(3 k + 3) s} + p^{-(3 k + 4) s}\right)}\right). 
  \end{align*}
  For $\sigma > 1/(3 k)$, we see that
  \begin{align*}
    &\lr(1 - p^{-k \sigma})^{-2} \lr(1 - p^{-2 k \sigma})^{-1} \lr(1 - p^{-(2 k + 2) \sigma})^{-1}\\
    &\quad \times \left(\vphantom{2 + 4 p^{-\sigma} + 2 p^{-2 \sigma} + p^{-k \sigma} + 3 p^{-(k + 2) \sigma} + 2 p^{-(k + 3) \sigma} + p^{-(3 k + 2)} + 2 p^{-(3 k + 3) \sigma} + p^{-(3 k + 4) \sigma}} 2 + 4 p^{-\sigma} + 2 p^{-2 \sigma} + p^{-k \sigma} + 3 p^{-(k + 2) \sigma} \right.\\
    &\quad \phantom{{} \times \left(\vphantom{2 + 4 p^{-\sigma} + 2 p^{-2 \sigma} + p^{-k \sigma} + 3 p^{-(k + 2) \sigma} + 2 p^{-(k + 3) \sigma} + p^{-(3 k + 2)} + 2 p^{-(3 k + 3) \sigma} + p^{-(3 k + 4) \sigma}}\right.} \quad \left. {} + 2 p^{-(k + 3) \sigma} + p^{-(3 k + 2)} + 2 p^{-(3 k + 3) \sigma} + p^{-(3 k + 4) \sigma} \vphantom{2 + 4 p^{-\sigma} + 2 p^{-2 \sigma} + p^{-k \sigma} + 3 p^{-(k + 2) \sigma} + 2 p^{-(k + 3) \sigma} + p^{-(3 k + 2)} + 2 p^{-(3 k + 3) \sigma} + p^{-(3 k + 4) \sigma}}\right)\\
    &\quad \ll 1. 
  \end{align*}
  The above immediately gives Lemma~\ref{lem:decomp3}. 
\end{proof}
\begin{lem}[{\cite[Lemma~3.4]{C}}]\label{lem:cohen}
  For any real number $x \ge 0$ and any integer $q \ge 1$, we have
  \begin{equation}
    \sum_{\substack{n \le x \\ (n, q) = 1}}1 = \frac{\varphi(q)}{q} x + \O(\theta(q)). \label{eq:cohen}
  \end{equation}
\end{lem}
\begin{lem}\label{lem:div-avg}
  For any real number $x \ge 1$, we have
  \begin{equation}
    \sum_{n \le x} \tau(n) \ll x \log(1 + x). 
  \end{equation}
\end{lem}
\begin{proof}
  This lemma follows from \cite[Theorem~3.3]{A1}. 
\end{proof}
\begin{lem}\label{lem:coprime}
  For any real number $x \ge 1$, any integer $q \ge 1$ and any multiplicative function $f$, we have
  \begin{equation}
    \sum_{\substack{n \le x \\ (n, q) = 1}} f(n) = \sum_{\substack{d \le x \\ \rad(d) \mid n}} f^{\ast(-1)}(d) \sum_{l \le x/d} f(l) \label{eq:coprime}
  \end{equation}
  where $f^{\ast(-1)}$ is the Dirichlet inverse function of $f$ and
  \begin{equation}
    \rad(d) \coloneq \prod_{p \mid d} p. \label{eq:rad}
  \end{equation}
\end{lem}
\begin{proof}
  For any integer $n \ge 1$, we see that
  \begin{equation}
    f(n) = f\lr(\lr(n, q^{n}) \frac{n}{\lr(n, q^{n})}) = f\lr(\lr(n, q^{n})) f\lr(\frac{n}{\lr(n, q^{n})}) = \sum_{\substack{a b = n \\ \rad(a) \mid q \\ (b, q) = 1}} f(a) f(b) \label{eq:coprime-1}
  \end{equation}
  and
  \begin{equation}
    \sum_{\substack{d l = n \\ \rad(d) \mid q \\ \rad(l) \mid q}} f^{\ast(-1)}(d) f(l) = \sum_{\substack{d l = n \\ \rad(n) \mid q}} f^{\ast(-1)}(d) f(l) = \floor{\frac{1}{n}}. \label{eq:coprime-2}
  \end{equation}
  Combinig \eqref{eq:coprime-1} and \eqref{eq:coprime-2}, we obtain
  \begin{equation}
    \begin{aligned}
      \sum_{\substack{d l = n \\ \rad(d) \mid q}} f^{\ast(-1)}(d) f(l) &= \sum_{\substack{d a b = n \\ \rad(d) \mid q \\ \rad(a) \mid q \\ (b, q) = 1}} f^{\ast(-1)}(d) f(a) f(b)\\
      &= \sum_{\substack{m b = n \\ (b, q) = 1}} \floor{\frac{1}{m}} f(b)\\
      &= f(n) \floor{\frac{1}{(n, q)}}. 
    \end{aligned}\label{eq:coprime-3}
  \end{equation}
  Summing up \eqref{eq:coprime-3} with respect to all positive integer $n \le x$ concludes the proof. 
\end{proof}
\begin{lem}\label{lem:mmu}
  For any real number $x \ge 1$, we have
  \begin{equation}
    \sum_{n \le x}\mmu(n) \ll x \dlt(x) \label{eq:mmu}
  \end{equation}
  where
  \[
  \dlt(x) \coloneq \exp(-\frac{D_{1} \lr(\log(x + e^{e^{1/3}}))^{3/5}}{\displaystyle \lr(\log\log(x + e^{e^{1/3}}))^{1/5}})
  \]
  and $D_{1}$ is an absolute constant. 
\end{lem}
\begin{proof}
  See \cite[Lemma~4.4]{MY}. 
\end{proof}
\begin{note}
  $\delta(x)$ is decreasing function. The inequality
  \begin{equation}
    \dlt(xy) \ge \dlt(x) \dlt(y)
  \end{equation}
  holds for any real numbers $x \ge 0$ and $y \ge 0$. 
\end{note}
\begin{lem}\label{lem:mmuq}
  For any real number $x \ge 1$ and any integer $q \ge 1$, we have
  \begin{equation}
    \sum_{\substack{n \le x \\ (n, q) = 1}} \mmu(n) \ll \theta(q) x \dlt(x). 
  \end{equation}
\end{lem}
\begin{proof}
  By Lemma~\ref{lem:coprime} and Lemma~\ref{lem:mmu}, we obtain
  \begin{align*}
    \sum_{\substack{n \le x \\ (n, q) = 1}} \mmu(n) &= \sum_{\substack{d \le x \\ \rad(d) \mid q}} \tau(d) \sum_{\substack{l \le x/d}} \mmu(l)\\
    &\ll \sum_{\substack{d \le x \\ \rad(d) \mid q}} \tau(d) \frac{x}{d} \dlt(\frac{x}{d})\\
    &\le x \dlt(x) \sum_{\rad(d) \mid q} \frac{\tau(d)}{d \dlt(d)}\\
    &\ll x \dlt(x) \sum_{\rad(d) \mid q} \frac{\tau(d)}{d^{1/2}}\\
    &= \frac{q}{\J^{2}{1/2}{q}} x \dlt(x)\\
    &\ll \theta(q) x \dlt(x). 
  \end{align*}
\end{proof}
\begin{lem}\label{lem:mmuqk}
  For any real number $x \ge 1$ and any integers $k \ge 2$ and $q \ge 1$, we have
  \begin{equation}
    \sum_{\substack{n \le x \\ (n, q) = 1}} \frac{\mmu(n)}{n^{k}} = \frac{q^{k}}{\zeta^{2}(k) \J^{2}{k}{q}} + \O(\theta(q) x^{-k + 1} \dlt(x)). 
  \end{equation}
\end{lem}
\begin{proof}
  Applying Abel summation formula and Lemma~\ref{lem:mmuq}, we obtain
  \begin{align*}
    \sum_{\substack{n \le x \\ (n, q) = 1}} \frac{\mmu(n)}{n^{k}} &= \frac{q^{k}}{\zeta^{2}(k) \J^{2}{k}{q}} - \sum_{\substack{n > x \\ (n, q) = 1}} \frac{\mmu(n)}{n^{k}}\\
    &= \frac{q^{k}}{\zeta^{2}(k) \J^{2}{k}{q}} + x^{-k} \sum_{\substack{n \le x \\ (n, q) = 1}} \mmu(n)\\
    &\phantom{{}={}} \quad - k \int^{\infty}_{x} t^{-k - 1} \sum_{\substack{n \le t \\ (n, q) = 1}} \mmu(n) \dd{t}\\
    &= \frac{q^{k}}{\zeta^{2}(k) \J^{2}{k}{q}} + \O(\theta(q) x^{-k + 1} \dlt(x)). 
  \end{align*}
\end{proof}
\begin{lem}\label{lem:zetapoly}
  For any real numbers $x \ge 1$ and $0 < \sigma < 1$, we have
  \begin{equation}
    \sum_{n \le x} n^{-\sigma} \ll \frac{x^{1 - \sigma}}{1 - \sigma}. \label{eq:zetapoly}
  \end{equation}
\end{lem}
\begin{proof}
  This lemma follows from \cite[Theorem~3.2~(b)]{A1}. 
\end{proof}
\begin{lem}\label{lem:ck1}
  For any real number $x \ge 1$ and any integer $k \ge 2$, we have
  \[
  \sum_{n \le x} \abs{c_{k}(n)} \ll x^{1/(k +1)}. 
  \]
\end{lem}
\begin{proof}
  By \eqref{eq:ckbk}, we see that
  \[
  \abs{c_{k}(n)} \le \sum_{d^{k + 1} \mid n} \abs{b_{k}\lr(\frac{n}{d^{k + 1}})}. 
  \]
  Hence by \eqref{eq:absbk}, we obtain
  \begin{align*}
    \sum_{n \le x} \abs{c_{k}(n)} &\le \sum_{l \le x} \abs{b_{k}(l)} \sum_{d \le (x/l)^{1/(k + 1)}} 1\\
    &\le x^{1/(k + 1)} \sum_{l \le x} \frac{\abs{b_{k}(l)}}{l^{1/(k + 1)}}\\
    &\ll x^{1/(k + 1)}. 
  \end{align*}
\end{proof}
\begin{lem}\label{lem:ck2}
  For any real number $x \ge 1$ and any integers $k \ge 2$ and $q \ge 1$, we have
  \begin{equation}
    \sum_{\substack{n \le x \\ (n, q) = 1}} \frac{c_{k}(n)}{n} = \frac{\zeta^{2}(k) \J^{2}{k}{q} A_{k, q}}{\varphi(q) q^{2 k - 1}} + \O(x^{-1 + 1/(k + 1)}). \label{eq:ck1}
  \end{equation}
\end{lem}
\begin{proof}
  Applying Abel summation formula and Lemma~\ref{lem:ck1} easily completes the proof. 
\end{proof}
\begin{lem}\label{lem:ck3}
  For any real number $x \ge 1$ and any integer $k \ge 2$, we have
  \[
  \sum_{n \le x} \frac{\abs{c_{k}(n)}}{n^{1/(k + 1)}} \ll \log(1 + x^{1/k}). 
  \]
\end{lem}
\begin{proof}
  We again apply Abel summation formula and Lemma~\ref{lem:ck1}. 
\end{proof}
\begin{lem}[{\cite[Corollary~5.3]{MV}}]\label{lem:mv-perron}
  For any arithmetic function $\lr(a_{n})_{n \ge 1}$ and any real numbers $\sigma_{0} > \max\{0, \sigma_{\mathrm{a}}\}$, $x > 0$ and $T > 0$, we have
  \begin{align*}
    \sideset{}{^\prime}{\sum}_{n \le x} a_{n} &= \frac{1}{2 \pi i} \int^{\sigma_{0} + iT}_{\sigma_{0} - iT} \sum_{n = 1}^{\infty} \frac{a_{n}}{n^{s}} \frac{x^{s}}{s} \dd{s}\\
    &\phantom{{}={}} \quad + \O(\sum_{\substack{x/2 < n < 2 x \\ n \ne x}} \abs{a_{n}} \min\{1, \frac{x}{T \abs{x - n}}\})
  \end{align*}
  where $\sigma_{\mathrm{a}}$ is the abscissa of absolute convergence of the Dirichlet series of $\lr(a_{n})_{n \ge 1}$ and
  \[
  \sideset{}{^\prime}{\sum}_{n \le x} = \frac{1}{2} \lr(\sum_{n < x} + \sum_{n \le x}). 
  \]
\end{lem}
\begin{lem}\label{lem:omega}
  For any real numbers $c \ge 0$, $\theta > 0$ and $x \ge 0$, we define
  \begin{align*}
    \omg{c}[\theta]{x} &\coloneq \exp(\frac{c \log(x + e^{e^{2}})}{\displaystyle \log\log(x + e^{e^{2}})} \max\{1, \log\frac{e}{\displaystyle \theta \log\log(x + e^{e^{2}})}\}), \\
    \omg{c}{x} &\coloneq \exp(\frac{c \log(x + e^{e^{2}})}{\displaystyle \log\log(x + e^{e^{2}})}). 
  \end{align*}
  Then the following holds. 
  \begin{enumerate}
    \item The function $\omg{c}[\theta]{x}$ is increasing in $c$ and $x$ and decreasing in $\theta$. 
    \item For any real numbers $c \ge 0$, $\theta > 0$ and $x \ge 0$, we have
    \[
    \omg{c}[\theta]{x}=\omg{c}[\min\{\theta, \frac{1}{\displaystyle \log\log(x + e^{e^{2}})}\}]{x}. 
    \]
    In particular, we have
    \[
    \omg{c}[\theta]{x} = \omg{c}{x}
    \]
    if $\theta \ge 1/\log\log(x + e^{e^{2}})$. 
    \item For any real numbers $c \ge 0$, $\theta > 0$, $x \ge 0$ and $y \ge 0$, we have
    \begin{align*}
      \omg{c}[\theta]{x + y} &\le \omg{c}[\theta]{x}\omg{c}[\theta]{y}, \\
      \omg{c}[\theta]{x y} &\le \omg{c}[\theta]{x}\omg{c}[\theta]{y}. 
    \end{align*}
    \item For any real numbers $c \ge 0$, $\theta > 0$, $x \ge 0$ and $a \ge 1$, we have
    \[
    \omg{c}[\frac{\theta}{a}]{x} \le \omg{c(1+\log a)}[\theta]{x}. 
    \]
    \item For any real numbers $x \ge 0$ and $\theta > 0$, we have
    \[
    \log(1+\frac{x}{\theta}) \ll \omg{1}[\theta]{x}. 
    \]
  \end{enumerate}
\end{lem}
All the results in Lemma~\ref{lem:omega} can be proven easily, hence we omit the proof. 
\begin{lem}\label{lem:div-max}
  For any integer $n \ge 1$, we have
  \begin{equation}
    \tau(n) \ll \omg{1}{n}. \label{eq:div-max}
  \end{equation}
\end{lem}
\begin{proof}
  This lemma follows from \cite[Theorem~I.5.4]{T}. 
\end{proof}
\begin{lem}\label{lem:mv-1323}
  Assume that the Riemann Hypothesis is true. For $\sigma > 1/2$, 
  \begin{equation}
    \abs{\zeta^{-1}(s)} \le \omg{C_{1}}[\sigma - \frac{1}{2}]{\abs{t}}
  \end{equation}
  holds. Here $C_{1}$ is an absolute positive constant. 
\end{lem}
\begin{proof}
  This lemma follows from \cite[Theorem~13.23]{MV}. 
\end{proof}
\begin{lem}\label{lem:mv-1318}
  Assume that the Riemann Hypothesis is true. Then
  \begin{equation}
    \abs{\zeta(s)} \le \omg{C_{2}}{\abs{t}} \lr(1 + \frac{1}{\abs{s}})
  \end{equation}
  holds for any complex number $s$ with $\Re(s) \ge 1/2$ and $s \ne 1$. 
\end{lem}
\begin{proof}
  This lemma follows from \cite[Theorem~13.18]{MV}. 
\end{proof}
\begin{lem}\label{lem:zeta-left}
  Assume that the Riemann Hypothesis is true. We have
  \[
  \zeta(s) \ll \lr(1 + \abs{t})^{1/2 - \sigma} \omg{C_{2}}{\abs{t}}
  \]
  for $0 \le \sigma \le 1/2$. 
\end{lem}
\begin{proof}
  This lemma follows from \cite[Corollary~10.5, Theorem~13.18]{MV}. 
\end{proof}
\begin{lem}\label{lem:gy}
  Assume that the Riemann Hypothesis is true. For any real number $Y \ge 1$ and any integer $q \ge 1$, we have
  \begin{align*}
    g_{Y, q}(s) &\coloneq \zeta^{-2}(s) \prod_{p \mid q} \lr(1 - p^{-s})^{-2} - \sum_{\substack{n \le Y \\ (n, q) = 1}} \frac{\mmu(n)}{n^{s}}\\
    &\ll 2^{\sigma} \kappa(q) Y^{1/2 - \sigma} \omg{C_{3}}[\sigma - \frac{1}{2}]{Y^{\sigma}} \omg{C_{3}}[\sigma - \frac{1}{2}]{\abs{t}}
  \end{align*}
  for $\sigma > 1/2$. Here
  \[
  \kappa(q) \coloneq \prod_{p \mid q} \lr(1 - p^{-1/2})^{-2}
  \]
  and $C_{3}$ is an absolute positive constant. 
\end{lem}
\begin{proof}
  By Lemma~\ref{lem:mv-perron} and Lemma~\ref{lem:div-max}, we see that
  \[
  \sum_{\substack{n \le Y \\ (n, q) = 1}} \frac{\mmu(n)}{n^{s}} = \frac{1}{2 \pi i} \int^{\alpha + i T}_{\alpha - i T} \zeta^{-2}(s + w)\prod_{p \mid q} \lr(1 - p^{-(s + w)})^{-2} \frac{Y^{w}}{w} \dd{w} + R
  \]
  where
  \begin{align*}
    \alpha &= \frac{1}{2} + \frac{1}{\log(1 + Y)}\\
    T &= Y^{\sigma}
  \end{align*}
  and
  \begin{align*}
    R &\ll 2^{\sigma} Y^{-\sigma} \omg{1}{2 Y} \lr(1 + \frac{Y \log(1 + Y)}{T}) + \frac{Y^{1/2} \lr(\log(1 + Y))^{2}}{T}\\
    &\ll 2^{\sigma} Y^{1/2 - \sigma} \omg{3}{Y^{\sigma}}. 
  \end{align*}
  We rewrite the above integral as
  \begin{align*}
    &\frac{1}{2 \pi i} \int^{\alpha + i T}_{\alpha - i T} \zeta^{-2}(s + w)\prod_{p \mid q} \lr(1 - p^{-(s + w)})^{-2} \frac{Y^{w}}{w} \dd{w}\\
    &\quad = \frac{1}{2 \pi i} \left(\int^{\alpha + i T}_{\beta + i T} + \int^{\beta + i T}_{\beta - i T} + \int^{\beta - i T}_{\alpha - i T}\right) \zeta^{-2}(s + w)\prod_{p \mid q} \lr(1 - p^{-(s + w)})^{-2} \frac{Y^{w}}{w} \dd{w}\\
    &\quad \phantom{{}={}} \quad + \zeta^{-2}(s) \prod_{p \mid q} \lr(1 - p^{-s})^{-2}\\
    &\quad \eqcolon I_{1} + I_{2} + I_{3} + \zeta^{-2}(s) \prod_{p \mid q} \lr(1 - p^{-s})^{-2}
  \end{align*}
  where
  \[
  \beta = \frac{1}{2} - \sigma + \min\{\frac{1}{\displaystyle \log\log(\abs{t} + T + e^{e^{2}})}, \frac{1}{2} \lr(\sigma - \frac{1}{2})\}. 
  \]
  By Lemma~\ref{lem:mv-1323}, we have
  \begin{align*}
    I_{1} &\ll \omg{2 C_{1}}[\sigma + \beta - \frac{1}{2}]{\abs{t + T}} \kappa(q) \frac{Y^{\alpha}}{T}\\
    &\ll \kappa(q) Y^{1/2 - \sigma} \omg{4C}[\sigma - \frac{1}{2}]{\abs{t} + Y^{\sigma}}, \\
    I_{2} &\ll \omg{2 C_{1}}[\sigma + \beta - \frac{1}{2}]{\abs{t} + T} \kappa(q) Y^{\beta} \int^{T}_{-T}\frac{\dd{v}}{\abs{\beta} + \abs{v}}\\
    &\ll \kappa(q) \omg{4 C_{1}}[\sigma - \frac{1}{2}]{\abs{t} + Y^{\sigma}}\\
    &\phantom{{}\ll {}}\quad \times Y^{1/2 - \sigma} \exp(\frac{\log Y}{\displaystyle \log\log(\abs{t} + T + e^{e^{2}})}) \log(1 + \frac{T}{\abs{b}})\\
    &\ll \kappa(q) Y^{1/2 - \sigma} \omg{4 C_{1} + 4}[\sigma - \frac{1}{2}]{\abs{t} + Y^{\sigma}}
    \intertext{and}
    I_{3} &\ll \omg{4 C_{1}}[\sigma + \beta - \frac{1}{2}]{\abs{t - T}} \kappa(q) \frac{Y^{\alpha}}{T}\\
    &\ll \kappa(q) Y^{1/2 - \sigma} \omg{4 C_{1}}[\sigma - \frac{1}{2}]{\abs{t} + Y^{\sigma}}. 
  \end{align*}
  Combining the above estimates completes the proof. 
\end{proof}
\section{Proof of the Theorems}
\subsection{Proof of Theorem~\ref{thm:uncond}}
We begin by estimating the sum $\sum_{n \le x, (n, q) = 1} f_{k}(n)$. 
Let $z = x^{1/k}$ and $0 < \rho \le 1$ and split the sum as follows: 
\begin{align*}
  \sum_{\substack{n \le x \\ (n, q) = 1}} f_{k}(n) &= \sum_{\substack{d^{k} l \le x \\ (d, q) = (l, q) = 1}} \mmu(d)\\
  &= \lr(\sum_{\substack{d^{k} l \le x \\ d \le \rho z \\ (d, q) = (l, q) = 1}} + \sum_{\substack{d^{k} l \le x \\ l \le \rho^{-k} \\ (d, q) = (l, q) = 1}} - \sum_{\substack{d \le \rho z \\ l \le \rho^{-k} \\ (d, q) = (l, q) = 1}}) \mmu(d)\\
  &\eqcolon S_{1} + S_{2} - S_{3}. 
\end{align*}
By Lemma~\ref{lem:cohen}, Lemma~\ref{lem:mmuqk} and Lemma~\ref{lem:div-avg}, we see that
\begin{equation}
  \begin{split}
    S_{1} &= \sum_{\substack{d \le \rho z \\ (d, q) = 1}} \mmu(d) \sum_{\substack{l \le x/d^{k} \\ (l, q) = 1}} 1\\
    &= \frac{\varphi(q)}{q} x \sum_{\substack{d \le \rho z \\ (d, q) = 1}} \frac{\mmu(d)}{d^{k}} + \O(\theta(q) \sum_{\substack{d \le \rho z \\ (d, q) = 1}} \abs{\mmu(d)})\\
    &= \frac{\varphi(n) n^{2 k - 1}}{\zeta^{2}(k) \J^{2}{k}{n}} x + \O(\theta(q) \rho^{1 - k} z \dlt(\rho z)) +\O(\theta(q) \rho z \log(1 + \rho z)). 
  \end{split}\label{eq:s1}
\end{equation}
Applying Lemma~\ref{lem:mmuq} and Lemma~\ref{lem:zetapoly}, we obtain
\begin{equation}
  \begin{split}
    S_{2} &= \sum_{\substack{l \le \rho^{-k} \\ (l, q) = 1}} \sum_{\substack{d \le (x/l)^{1/k} \\ (d, q) = 1}} \mmu(d)\\
    &\ll \theta(q) x^{1/k} \sum_{\substack{l \le \rho^{-k} \\ (l, q) = 1}} l^{-1/k} \dlt(\lr(\frac{x}{l})^{1/k})\\
    &\ll \theta(q) \rho^{1 - k} z \dlt(\rho z). 
  \end{split}\label{eq:s2}
\end{equation}
Lemma~\ref{lem:mmuq} also gives
\begin{equation}
  S_{3} = \sum_{\substack{d \le \rho z \\ (d, q) = 1}} \mmu(d) \sum_{\substack{l \le \rho^{-k} \\ (l, q) = 1}} 1 \ll \theta(q) \rho^{1 - k} z \dlt(\rho z). \label{eq:s3}
\end{equation}
Putting $\rho = \dlt^{1/k}(z \dlt^{1/k}(z))$, we have
\begin{equation}
  \begin{split}
    \rho^{1 - k} z \dlt(\rho z) &= \rho z \dlt^{-1}(z \dlt^{1/k}(z)) \dlt(z \dlt^{1/k}(z \dlt^{1/k}(z)))\\
    &\le \rho z \dlt^{-1}(z \dlt^{1/k}(z)) \dlt(z \dlt^{1/k}(z))\\
    &= \rho z\\
    &= z \dlt^{1/k}(z \dlt^{1/k}(z))\\
    &\le z \dlt^{1/k}(z^{1/2})\\
    &\le z \dlt^{1/(2k)}(z). 
  \end{split}\label{eq:rhoz}
\end{equation}
Combining \eqref{eq:s1}, \eqref{eq:s2}, \eqref{eq:s3} and \eqref{eq:rhoz}, we find that
\[
\sum_{\substack{n \le x \\ (n, q) = 1}} f_{k}(n) = \frac{\varphi(q) q^{2 k - 1}}{\zeta^{2}(k) \J^{2}{k}{n}} x + \O(\theta(q) x^{1/k}\dlt^{1/(2k)}(x^{1/k})\log(1+x^{1/k})). 
\]
Therefore, applying Lemma~\ref{lem:decomp1}, Lemma~\ref{lem:ck2} and Lemma~\ref{lem:ck3}, we obtain
\begin{align*}
  \sum_{\substack{n \le x \\ (n, q) = 1}} \mu_{k}(n) &= \sum_{\substack{d \le x \\ (d, n) = 1}} c_{k}(d) \sum_{\substack{l \le x/d \\ (l, n) = 1}} f_{k}(l)\\
  &= \frac{\varphi(q) q^{2 k - 1}}{\zeta^{2}(k) \J^{2}{k}{n}} x \sum_{\substack{d \le x \\ (d, q) = 1}} \frac{c_{k}(d)}{d}\\
  &\phantom{{}={}} + O\left(\theta(q) x^{1/k} \dlt^{1/(2k)}(x^{1/k}) \log(1 + x^{1/k}) \vphantom{\sum_{\substack{d \le x \\ (d, q) = 1}} \frac{\abs{c_{k}(d)}}{d^{1/k} \dlt^{1/(2k)}(d^{1/k})}}\right.\\
  &\phantom{\phantom{{}={}} + O\left(\vphantom{\sum_{\substack{d \le x \\ (d, q) = 1}} \frac{\abs{c_{k}(d)}}{d^{1/k} \dlt^{1/(2k)}(d^{1/k})}}\right.} \quad \left. {}\times \sum_{\substack{d \le x \\ (d, q) = 1}} \frac{\abs{c_{k}(d)}}{d^{1/k} \dlt^{1/(2k)}(d^{1/k})}\right)\\
  &= A_{k, q} x + \O(\theta(q) x^{1/k} \dlt^{1/(2k)}(x^{1/k}) \lr(\log(1 + x^{1/k}))^{2})\\
  &= A_{k, q} x + \O_{k}(\theta(q) x^{1/k} \dlt^{1/(3 k)}(x^{1/k})). 
\end{align*}
Finally, if $x \ge 3$, we see that
\begin{align*}
  \dlt^{1/(3 k)}(x^{1/k})&= \exp(-\frac{D_{1}}{3 k} \frac{\lr(\log(x^{1/k} + e^{e^{1/3}}))^{3/5}}{\displaystyle \lr(\log\log(x^{1/k} + e^{e^{1/3}}))^{1/5}})\\
  &\le \exp(-\frac{D_{1}}{3 k} \frac{\lr(k^{-1} \log(x + e^{e^{1/3}}))^{3/5}}{\displaystyle \lr(\log\log(x + e^{e^{1/3}}))^{1/5}})\\
  &\le \exp(-\frac{D_{1}}{4} k^{-8/5} \frac{\lr(\log x)^{3/5}}{\lr(\log\log x)^{1/5}}), 
\end{align*}
and we are done. 
\subsection{Proof of Theorem~\ref{thm:RH}}
We first split the sum $\sum_{n \le x, (n, q) = 1} f_{k}(n)$ as follows: 
\begin{align*}
  \sum_{\substack{n \le x \\ (n, q) = 1}} f_{k}(n) &= \sum_{\substack{d^{k} l \le x \\ (d, q) = (l, q) = 1}} \mmu(d)\\
  &= \lr(\sum_{\substack{d \le Y \\ (d, q) = 1}} + \sum_{\substack{Y < d \le x^{1/k} \\ (d, q) = 1}}) \mmu(d) \sum_{\substack{l \le x/d^{k} \\ (l, q) = 1}} 1\\
  &\eqcolon S_{1} + S_{2}
\end{align*}
where $Y \ge 1$. 
By Lemma~\ref{lem:cohen} and Lemma~\ref{lem:div-avg}, we see that
\begin{align*}
  S_{1} &= \sum_{\substack{d \le Y \\ (d, q) = 1}} \mmu(d) \sum_{\substack{l \le x/d^{k} \\ (l, q) = 1}} 1\\
  &= \sum_{\substack{d \le Y \\ (d, q) = 1}} \mmu(d) \lr(\frac{\varphi(q)}{q} \frac{x}{d^{k}} + \O(\theta(q)))\\
  &= \frac{\varphi(q)}{q} x \sum_{\substack{d \le Y \\ (d, q) = 1}} \frac{\mmu(d)}{d^{k}} + \O(\theta(q) Y \lr(1 + \log Y)). 
\end{align*}
Since
\[
\abs{\sum_{\substack{d^{k} \mid n \\ d > Y}} \mmu(d)} \le \sum_{d^{k} \mid n} \tau(d) \le \sum_{d \mid n} \tau(d) \le \sum_{d \mid n} \tau(n) = \tau^{2}(n), 
\]
by Lemma~\ref{lem:mv-perron} and Lemma~\ref{lem:div-max}, we obtain
\begin{equation}
  \begin{split}
    S_{2} &= \sum_{\substack{Y < d \le x^{1/k} \\ (d, q) = 1}} \mmu(d) \sum_{\substack{l \le x/d^{k} \\ (l, q) = 1}} 1 \\
    &= \frac{1}{2 \pi i} \int^{\alpha + i T}_{\alpha - i T} \zeta(s) \prod_{p \mid q} \lr(1 - p^{-s}) g_{Y, q}(ks) \frac{x^{s}}{s} \dd{s} + R
  \end{split} \label{eq:s2-RH}
\end{equation}
where
\begin{align*}
  \alpha &= 1 + \frac{1}{\log(1 + x)}, \\
  T &= x
\end{align*}
and
\begin{align*}
  R &\ll \omg{2}{2 x} \lr(1 + \frac{x \log(1 + x)}{T}) + \frac{x \lr(\log(1 + x))^{3}}{T}\\
  &\ll \omg{3}{x}. 
\end{align*}
We move the line segment $[\alpha - i T, \alpha + i T]$ which is the contour for the first term on the right hand side of \eqref{eq:s2-RH} to the contour consisting of the line segments $[\alpha - i T, 1/2 - i T]$, $[1/2 - i T, 1/2 + i T]$ and $[1/2 + i T, \alpha +i T]$. 
Then we have
\begin{align*}
  &\frac{1}{2 \pi i} \int^{\alpha + i T}_{\alpha - i T} \zeta(s) \prod_{p \mid q} \lr(1 - p^{-s}) g_{Y, q}(ks) \frac{x^{s}}{s} \dd{s}\\
  &\quad = \frac{\varphi(q)}{q} x \lr(\frac{q^{2 k}}{\zeta^{2}(k) \J^{2}{k}{q}} - \sum_{\substack{n \le Y \\ (n, q) = 1}} \frac{\mmu(n)}{n^{k}}) + I_{1} + I_{2} + I_{3}
\end{align*}
where $I_{1}$ and $I_{3}$ denote the integrals over the horizontal line segments and $I_{2}$ is the integral over the vertical line segment. 
Applying Lemma~\ref{lem:mv-1318} and Lemma~\ref{lem:gy}, we have
\begin{align*}
  I_{1}, I_{3} &\ll \omg{C_{2}}{T} \eta(q) (2 e)^{k} \kappa(q) Y^{(1 - k)/2} \omg{3 C_{3}}{Y^{k}} \omg{C_{3}}{k} \omg{C_{3}}{T}\\
  &\ll \theta(q) \exp(\lr(2 + C_{3}) k) Y^{(1 - k)/2} \omg{3 C_{3}}{Y^{k}} \omg{C_{2} + C_{3}}{x}
\end{align*}
and
\begin{align*}
  I_{2} &\ll \omg{C_{2}}{T} \eta(q) 2^{k} \kappa(q) Y^{(1 - k)/2} \omg{C_{3}}{Y^{k}} \omg{C_{3}}{k} \omg{C_{3}}{T} x^{1/2} \log(1 + T)\\
  &\ll \theta(q) \exp(\lr(1 + C_{3}) k) x^{1/2} \omg{C_{2} + C_{3} + 1}{x} Y^{(1 - k)/2} \omg{C_{3}}{Y^{k}}
\end{align*}
where
\[
\eta(q) \coloneq \prod_{p \mid q} \lr(1 + p^{-1/2}). 
\]
Putting $Y = \exp(2 C_{3} + 4) x^{1/(k + 1)}$, we obtain
\[
\sum_{\substack{n \le x \\ (n, q) = 1}} f_{k}(n) = \frac{\varphi(q) q^{2 k - 1}}{\zeta^{2}(k) \J^{2}{k}{q}} x + \O(\theta(q) x^{1/(k + 1)} \omg{C_{2} + 4 C_{3} + 3}{x}). 
\]
Here we note that $S_{2} = 0$ if $Y \ge x^{1/k}$. 
Finally, by Lemma~\ref{lem:decomp1}, Lemma~\ref{lem:ck2} and Lemma~\ref{lem:ck3}, we obtain
\begin{align*}
  \sum_{\substack{n \le x \\ (n, q) = 1}} \mu_{k}(n) &= \sum_{\substack{d \le x \\ (d, q) = 1}} c_{k}(d) \sum_{\substack{l \le x/d \\ (l, q) = 1}} f_{k}(l)\\
  &= \frac{\varphi(q) q^{2 k - 1}}{\zeta^{2}(k) \J^{2}{k}{q}} x \sum_{\substack{d \le x \\ (d, q) = 1}} \frac{c_{k}(d)}{d}\\
  &\phantom{{}={}} \quad + \O(\theta(q) x^{1/(k + 1)} \omg{C_{2} + 4 C_{3} + 3}{x} \sum_{d \le x} \frac{\abs{c_{k}(d)}}{d^{1/(k + 1)}})\\
  &= A_{k, q} x + \O(\theta(q) x^{1/(k + 1)} \omg{C_{2} + 4 C_{3} + 4}{x}). 
\end{align*}
Since
\[
\frac{\log x}{\log\log x} \ll \frac{\log(x + e^{e^{2}})}{\displaystyle \log\log(x + e^{e^{2}})}
\]
for $x \ge 3$, this completes the proof. 
\subsection{Proof of Theorem~\ref{thm:sq}}
By Lemma~\ref{lem:mv-perron} and Lemma~\ref{lem:decomp3}, we see that
\begin{equation}
  \begin{split}
    \sideset{}{^\prime}{\sum}_{\substack{n \le x \\ (n, q) = 1}} \mu_{k}(n) &= \frac{1}{2 \pi i} \int^{\alpha + i V}_{\alpha - i V} \frac{\zeta(s) \zeta((k + 1) s)}{\displaystyle \zeta^{2}(k s) \zeta(2 k s) \zeta((2 k + 2) s)}\\
    &\hphantom{{} = \frac{1}{2 \pi i}\int} \times \prod_{p \mid q} \frac{\lr(1 - p^{-s}) \lr(1 - p^{-(k + 1) s})}{\displaystyle \lr(1 - p^{-k s})^{2} \lr(1 - p^{-2 k s}) \lr(1 - p^{-(2 k + 2) s})}\\
    &\hphantom{{} = \frac{1}{2 \pi i}\int} \times \lr(\sum_{\substack{n = 1 \\ (n, q) = 1}}^{\infty} \frac{1}{n^{(2 k + 1) s}})^{2} D_{k, q}(s) \frac{x^{s}}{s} \dd{s}\\
    &\hphantom{{} = {}} \quad + \O(\frac{x \log(1 + x)}{V \norm{x}^{\prime}})
  \end{split} \label{eq:s-sq}
\end{equation}
where
\begin{align*}
  \alpha &= 1 + \frac{1}{\log(1 + x)}, \\
  V &\ge 1, \\
  D_{k, q}(s) &\coloneq \sum_{\substack{n = 1 \\ (n, q) = 1}}^{\infty} \frac{d_{k}(n)}{n^{s}}\\
    &= \prod_{p \nmid q} \frac{\lr(1 - 2 p^{-k s} + p^{-(k + 1)s}) \lr(1 - p^{-(k + 1) s}) \lr(1 - p^{-(2 k + 1) s})^{2}}{\lr(1 - p^{-k s})^{2} \lr(1 - p^{-2 k s}) \lr(1 - p^{-(2 k + 2) s})}
\end{align*}
for $\Re(s) > 1/(3 k)$, 
\[
\norm{x}^{\prime} \coloneq \min_{\substack{n \in \ZZ \\ n \ne x}}\abs{x - n}
\]
and
\[
\sideset{}{^\prime}{\sum}_{\substack{n \le x \\ (n, q) = 1}} \coloneq \frac{1}{2} \lr(\sum_{\substack{n \le x \\ (n, q) = 1}} + \sum_{\substack{n < x \\ (n, q) = 1}}). 
\]
We move the line segment $[\alpha - i V, \alpha + i V]$ which is the contour for the first term on the right hand side of \eqref{eq:s-sq} to the contour consisting of the line segments $[\alpha - i V, \beta - i V]$, $[\beta - i V, \beta + i V]$ and $[\beta + i V, \alpha + i V]$ with $1/(2 k) < \beta \le 3/(5 k)$. 
We denote the integrals over the horizontal line segments by $I_{1}$ and $I_{3}$, and the integral over the vertical line segment by $I_{2}$. 
Then we have
\begin{align*}
  &\frac{1}{2 \pi i} \int^{\alpha + i V}_{\alpha - i V} \frac{\zeta(s) \zeta((k + 1) s)}{\displaystyle \zeta^{2}(k s) \zeta(2 k s) \zeta((2 k + 2) s)} \\
  &\hphantom{\frac{1}{2 \pi i} \int^{\alpha + i V}_{\alpha - i V}} \quad \times \prod_{p \mid q} \frac{\lr(1 - p^{-s}) \lr(1 - p^{-(k + 1) s})}{\displaystyle \lr(1 - p^{-k s})^{2} \lr(1 - p^{-2 k s}) \lr(1 - p^{-(2 k + 2) s})}\\
  &\hphantom{\frac{1}{2 \pi i} \int^{\alpha + i V}_{\alpha - i V}} \quad \times \lr(\sum_{\substack{n = 1 \\ (n, q) = 1}}^{\infty} \frac{1}{n^{(2 k + 1) s}})^{2} D_{k, q}(s) \frac{x^{s}}{s} \dd{s}\\
  &\quad = A_{k, q} x + B_{k, q} x^{1/(k + 1)} + I_{1} + I_{2} + I_{3}. 
\end{align*}
Applying Lemma~\ref{lem:mv-1318}, Lemma~\ref{lem:zeta-left}, Lemma~\ref{lem:mv-1323} and Lemma~\ref{lem:decomp3}, we get
\begin{align*}
  I_{1}, I_{3} &\ll V^{1/2} \omg{C_{2}}{V} \omg{C_{2}}{k V} \omg{2 C_{1}}[k \beta - \frac{1}{2}]{k V} \omg{2 C_{1}}{k V} C_{k, q} k^{2} \frac{x}{V}\\
  &\longrightarrow 0
\end{align*}
as $V \to +\infty$, where
\[
C_{k, q} \coloneq \prod_{p \mid q} \frac{\lr(1 + p^{-1/(2 k)}) \lr(1 + p^{-(k + 1)/(2 k)})}{\displaystyle \lr(1 - p^{-1/2})^{2} \lr(1 - p^{-1}) \lr(1 - p^{-(k + 1)/k})}. 
\]
Hence, we obtain
\begin{align*}
   \widehat{\Xi}_{k, q}(x) &= \frac{1}{2 \pi i} \int^{\beta + i \infty}_{\beta - i \infty} \frac{\zeta(s) \zeta((k + 1) s)}{\displaystyle \zeta^{2}(k s) \zeta(2 k s) \zeta((2 k + 2) s)} \\
  &\hphantom{{} = \frac{1}{2 \pi i} \int^{\beta + i \infty}_{\beta - i \infty}} \quad \times \prod_{p \mid q} \frac{\lr(1 - p^{-s}) \lr(1 - p^{-(k + 1) s})}{\displaystyle \lr(1 - p^{-k s})^{2} \lr(1 - p^{-2 k s}) \lr(1 - p^{-(2 k + 2) s})} \\
  &\hphantom{{} = \frac{1}{2 \pi i} \int^{\beta + i \infty}_{\beta - i \infty}} \quad \times \lr(\sum_{\substack{n = 1 \\ (n, q) = 1}}^{\infty} \frac{1}{n^{(2 k + 1) s}})^{2} D_{k, q}(s) \frac{x^{s}}{s} \dd{s}
\end{align*}
where
\[
\widehat{\Xi}_{k, q}(x) \coloneq \sideset{}{^\prime}{\sum}_{\substack{n \le x \\ (n, q) = 1}} \mu_{k}(n) - A_{k, q} x - B_{k, q} x^{1/(k + 1)}. 
\]
By (A.3) and (A.5) in \cite{I}, we see that
\begin{equation}
  \begin{split}
    &\frac{1}{2 \pi} \int^{+\infty}_{-\infty} \frac{\abs{\zeta(\beta + i t)}^{2} \abs{\zeta((k + 1)(\beta + i t))}^{2}}{\abs{\zeta(k (\beta + i t))}^{4} \abs{\zeta(2 k (\beta + i t))}^{2} \abs{\zeta((2 k + 2)(\beta + i t))}^{2} \abs{\beta + i t}^{2}} \\
    &\hphantom{\frac{1}{2 \pi} \int^{+\infty}_{-\infty}} \quad \times \prod_{p \mid q} \frac{\abs{1 - p^{-(\beta + i t)}}^{2} \abs{1 - p^{-(k + 1)(\beta + i t)}}^{2}}{\abs{1 - p^{-k (\beta + i t)}}^{4} \abs{1 - p^{-2 k (\beta + i t)}}^{2} \abs{1 - p^{-(2 k + 2)(\beta + i t)}}^{2}} \\
    &\hphantom{\frac{1}{2 \pi} \int^{+\infty}_{-\infty}} \quad \times \abs{\sum_{\substack{n = 1 \\ (n, q) = 1}}^{\infty} \frac{1}{n^{(2 k + 1)(\beta + i t)}}}^{4} \dd{t} \\
    &\quad = \int^{+\infty}_{0} \widehat{\Xi}_{k, q}^{2}(x) x^{-1 - 2 \beta} \dd{x} \\
    &\quad = \int^{+\infty}_{0} \Xi_{k, q}^{2}(x) x^{-1 - 2 \beta} \dd{x}. 
  \end{split} \label{eq:f-sq}
\end{equation}
Using Lemma~\ref{lem:zeta-left}, Lemma~\ref{lem:mv-1318} and Lemma~\ref{lem:mv-1323}, we have
\begin{align*}
  &\frac{\abs{\zeta(\beta + i t)}^{2} \abs{\zeta((k + 1)(\beta + i t))}^{2}}{\abs{\zeta(k (\beta + i t))}^{4} \abs{\zeta(2 k (\beta + i t))}^{2} \abs{\zeta((2 k + 2)(\beta + i t))}^{2} \abs{\beta + i t}^{2}}\\
  &\quad \ll \lr(1 + \abs{t})^{1 - 2 \beta} \omg{2 C_{2}}{\abs{t}} \omg{2 C_{2}}{k \abs{t}} \omg{4 C_{1}}[k \beta - \frac{1}{2}]{k \abs{t}} \omg{2 C_{1}}{k \abs{t}} \\
  &\hphantom{\quad \ll {}} \quad \times \omg{2 C_{1}}{k \abs{t}} \frac{k^{2}}{\lr(1 + \abs{t})^{2}} \\
  &\quad \ll k^{2} \lr(1 + \abs{t})^{-1 - 2 \beta} \omg{8 C_{1} + 4 C_{2}}[k \beta - \frac{1}{2}]{k \abs{t}}. 
\end{align*}
Therefore, we find that
\begin{align*}
  \int^{+\infty}_{0} \Xi_{k, q}^{2}(x) x^{-1 - 2 \beta} \dd{x} &\ll k^{6} C_{k, q}^{2} \int_{0}^{+\infty} \lr(t + e^{e^{2}})^{-1 - 2 \beta} \omg{C}[k \beta - \frac{1}{2}]{k t} \dd{t} \\
  &\ll k^{7} C_{k, q}^{2} \int_{0}^{+\infty} \lr(k t + e^{e^{2}})^{-1 - 2 \beta} \omg{C}[k \beta - \frac{1}{2}]{k t} \dd{t} \\
  &= k^{6} C_{k, q}^{2} \int_{0}^{+\infty} U^{-1 - 2 \beta} \omg{C}[k \beta - \frac{1}{2}]{u} \dd{u}
\end{align*}
where $C = 8 C_{1} + 4 C_{2}$ and $U = u + e^{e^{2}}$. 
Letting
\[
y = \exp(\exp(2 C k \log\frac{e}{2 k \beta - 1})), 
\]
we see that
\[
\frac{C}{\log\log Y} \max\{1, \log\frac{e}{\lr(k \beta - 1/2) \log\log Y}\} \le \beta
\]
where $Y = y + e^{e^{2}}$. 
Thus, we obtain
\begin{align*}
  \int^{y}_{0} U^{-1 - 2 \beta} \omg{C}[k \beta - \frac{1}{2}]{u} \dd{u} &\ll k \omg{C}[k \beta - \frac{1}{2}]{y}\\
  &\ll k y^{1/k}
\end{align*}
and
\begin{align*}
  &\int^{+\infty}_{y} U^{-1 - 2 \beta} \omg{C}[k \beta - \frac{1}{2}]{u} \dd{u} \\
  &\quad \ll \int^{+\infty}_{y} \exp\left(\vphantom{\log U \left(-1 - 2 \beta + \frac{C}{\log\log Y} \times \max\{1, \log\frac{e}{\lr(k \beta - 1/2) \log\log Y}\}\right)} \log U \left(\vphantom{-1 - 2 \beta + \frac{C}{\log\log Y} \times \max\{1, \log\frac{e}{\lr(k \beta - 1/2) \log\log Y}\}} -1 - 2 \beta + \frac{C}{\log\log Y} \right.\right. \\
  &\hphantom{\quad \ll \int^{+\infty}_{y} \exp\left(\vphantom{\log U \left(-1 - 2 \beta + \frac{C}{\log\log Y} \times \max\{1, \log\frac{e}{\lr(k \beta - 1/2) \log\log Y}\}\right)}\right.} \quad \left.\left. {} \times \max\{1, \log\frac{e}{\lr(k \beta - 1/2) \log\log Y}\} \vphantom{-1 - 2 \beta + \frac{C}{\log\log Y} \times \max\{1, \log\frac{e}{\lr(k \beta - 1/2) \log\log Y}\}}\right) \vphantom{\log U \left(-1 - 2 \beta + \frac{C}{\log\log Y} \times \max\{1, \log\frac{e}{\lr(k \beta - 1/2) \log\log Y}\}\right)}\right) \dd{u}\\
  &\quad \ll k Y^{-1 - 2 \beta} \omg{C}[k \beta - \frac{1}{2}]{y} \\
  &\quad \ll k y^{1/k}. 
\end{align*}
Then we have
\[
\int^{+\infty}_{0} \Xi_{k, q}^{2}(x) x^{-1 - 2 \beta} \dd{x} \ll k^{7} C_{k, q}^{2} \exp(\frac{1}{k}\exp(2 C k \log\frac{e}{2 k \beta - 1}))
\]
immediately. 
Hence, we obtain
\begin{align*}
  \int^{T}_{1} \Xi_{k, q}^{2}(x) \dd{x} &= \int^{T}_{1} \Xi_{k, q}^{2}(x) x^{-1 - 2 \beta} x^{1 + 2 \beta} \dd{x} \\
  &\le T^{1 + 2 \beta} \int^{T}_{1} \Xi_{k, q}^{2}(x) x^{-1 - 2 \beta} \dd{x} \\
  &\ll T^{1 + 2 \beta} k^{7} C_{k, q}^{2} \exp(\frac{1}{k}\exp(2 C k \log\frac{e}{2 k \beta - 1})). 
\end{align*}
Setting
\[
\beta \coloneq \frac{1}{2 k} + \frac{1}{10 k} \lr(\log T)^{-1/(2 C k + 1)}, 
\]
we see that
\begin{align*}
  &T^{1 + 2 \beta} \exp(\frac{1}{k}\exp(2 C k \log\frac{e}{2 k \beta - 1})) \\
  &\quad = T^{1 + 1/k} \exp(\frac{1}{k} \lr(\frac{1}{5} + (5 e)^{2 C k}) \lr(\log T)^{2 C k/(2 C k + 1)}), 
\end{align*}
and the proof is done. 
\section*{Acknowledgement}
The author thanks Professor Isao Kiuchi for his kind guidance and valuable comments.

\medskip\noindent {\footnotesize Joint Graduate school of Mathematics for Innovation, Kyushu  University, Motooka 744, Nishiku-ku, Fukuoka 819-0395, Japan. \\
e-mail: \href{mailto:terada.reo.872@s.kyushu-u.ac.jp}{\texttt{terada.reo.872@s.kyushu-u.ac.jp}}}

\end{document}